\documentclass[12pt]{article} 

\usepackage{amsmath,amssymb,amsfonts}%
\usepackage{amsthm}%
\usepackage{xcolor}%
\usepackage{manyfoot}%

\def\A{\mathcal{A}}
\def\C{\mathbb{C}}
\def\H{\mathbb{H}}
\def\M{\mathcal{M}}
\def\MC{\mathcal{M}_{\rm C}}
\def\R{\mathbb{R}}
\def\Z{\mathbb{Z}}
\def\Sc{\operatorname{Sc}}

\def\coh{\operatorname{coh}}

\def\chcos{\sqrt{\cosh\eta-\cos\theta}}

\def\dom{\Omega_{\eta_0}}

\def\re{\operatorname{Re}}
\def\im{\operatorname{Im}}

\def\har{\operatorname{Har}}
\def\partialbar{\overline{\partial}}

\newcommand{\ind}[2]{{}^{\hspace{.1ex}#2}_{#1}}
\DeclareMathOperator{\Span}{Span}
\DeclareMathOperator{\Spanbar}{\overline{Span}}
 
\newtheorem{theorem}{Theorem}[section]%
\newtheorem{proposition}[theorem]{Proposition}%
\newtheorem{lemma}[theorem]{Lemma}
\newtheorem{definition}[theorem]{Definition}%

\begin{document}

\begin{center}
{\Large Harmonic and monogenic functions on toroidal domains }
\\[5ex]

 {Z.~Ashtab} \\
 Department of Mathematics, CINVESTAV-Quer\'etaro, Libramiento Norponiente \#2000, Fracc.~Real de Juriquilla, Santiago de Quer\'etaro, Qro., C.P.~76230 Mexico \\[1ex]

 {J.~Morais} \\
 Department of Mathematics, ITAM, R\'io Hondo~\#1, Col.~Progreso Tizap\'an, Mexico City, C.P.~01080 Mexico\\[1ex]

 {R.~Michael Porter} \\
 Department of Mathematics, CINVESTAV-Quer\'etaro, Libramiento Norponiente \#2000, Fracc.~Real de Juriquilla, Santiago de Quer\'etaro, Qro., C.P.~76230 Mexico \\[5ex]

   \parbox{.9\textwidth}{\textbf{Abstract.} A standard technique for producing monogenic functions is to apply the adjoint quaternionic Fueter operator to harmonic functions. We will show that this technique does not give a complete system in $L^2$ of a solid torus, where toroidal harmonics appear in a natural way. One reason is that this index-increasing operator fails to produce monogenic functions with zero index. Another reason is that the non-trivial topology of the torus requires taking into account a cohomology coefficient associated with monogenic functions, apparently not previously identified because it vanishes for simply connected domains. In this paper, we build a reverse-Appell basis of harmonic functions on the torus expressed in terms of classical toroidal harmonics. This means that the partial derivative of any element of the basis with respect to the axial variable is a constant multiple of another basis element with subscript increased by one. This special basis is used to construct respective bases in the real $L^2$-Hilbert spaces of reduced quaternion and quaternion-valued monogenic functions on toroidal domains.}

 \end{center}


 
\section*{Introduction}

Quaternionic analysis studies monogenic (also known as regular or hyperholomorphic) functions of a vector variable in a domain in three-dimensional Euclidean space, which are null solutions of generalized Cauchy-Riemann or Dirac-type systems \cite{BrackxDelangheSommen1982,GuerlebeckHabethaSproessig2008,GuerlebeckHabethaSproessig2016,Stein1970,SteinWeiss1971,Sudbery1979}. Since the Cauchy-Riemann and Dirac operators linearize the Laplacian (implying that each real component of a monogenic function is, therefore, Euclidean harmonic), its associated function theory provides the foundations to refine the theory of harmonic functions to three and four dimensions. Quaternionic analysis has become a central research area in mathematics, connecting with boundary value problems, partial differential equations theory, and other fields of physics and engineering \cite{GS1989,GS1997,KS1996,K2003}. As was noticed at an early stage and already implicit in R.\ Fueter's work \cite{Fueter1949}, monogenic functions are, in an essential way, characterized by power series expansions. Applications to solve boundary value problems that originate in mathematical physics often require manipulation of the series by termwise differentiation. For a simple comparison of coefficients, a desired property of the function system is that the derivative of each element is a constant multiple of another single element \cite{Appell1880}, nowadays called the ``Appell property.''

Considerable effort has been directed toward constructing complete orthogonal sets spanning the Hilbert spaces of square-integrable monogenic functions defined in the interior and exterior of the ball \cite{BockGuerlebeck2009,Cac2004,CGBHanoi2005,CacaoGuerlebeckBock2006,Cacao2010,MoraisHabilitation2021,Morais2023} on spheroidal domains  (both prolate and oblate) \cite{Morais2011,Morais2013,MoraisNguyenKou2016,Morais2019,MoraisHabilitation2021,Morais2023,NguyenGuerlebeckMoraisBock2014}, and on finite and infinite cylinders \cite{MoraisLe2011,MoraisKouLe2015}.  Central to the papers \cite{BockGuerlebeck2009} and \cite{Cac2004} is the construction of graded bases of reduced quaternion and quaternion-valued solid spherical monogenics satisfying an Appell-type property with respect to the hypercomplex derivative, while in \cite{MoraisNguyenKou2016} it was shown that in the case of an oblate spheroid, a basis cannot be both orthogonal and Appell. To the best of our knowledge, complete sets of monogenic functions have not previously been built in the context of the torus. In the present paper, we are concerned with developing a quaternionic function theory over toroidal domains and constructing bases of monogenic functions on these domains. The construction of explicit formulas for the monogenic Cauchy kernel and Green's formula on the torus is relatively known and can be found in \cite{Krausshar2001,KraussharRyan2005}.

In this paper, we pursue two lines of investigation which are, as we shall see, remarkably intertwined:

(i) We first construct a ``reverse-Appell basis" of interior toroidal harmonic functions. The classical toroidal harmonics do not satisfy a natural Appell property with respect to any of the variables. In fact, the application of $\partial/\partial x_0$ increases the index (or degree) of the toroidal harmonics rather than decreasing it, and it does not result in a multiple of another toroidal harmonic of the same type. Section \ref{sec:appell} establishes, by an appropriate change of basis, the existence of a collection of toroidal harmonic functions that satisfy what we may call a ``reverse Appell property." Although complicated for hand calculations, the reverse-Appell toroidal harmonic functions are easily handled from computational and theoretical considerations.

(ii) The reverse-Appell basis of toroidal harmonics is an essential element in our construction of the real Hilbert space of square-integrable reduced quaternion-valued monogenic functions on toroidal domains in Section \ref{Section_Basis_toroidal_monogenics1}. A basis of full quaternion-valued monogenic functions in the real Hilbert space $L^2$ over toroidal domains is derived in Section \ref{Section_Basis_toroidal_monogenics2}. Even though these structures are not closed under the quaternionic multiplication, considering a real-linear Hilbert space has its own importance. It turns out that the operators of some boundary value problems (e.g., Lam\'e system \cite{GS1997} and Stokes system \cite{GS21997}) are not $\mathbb{H}$-linear but are nevertheless efficiently treated by employing quaternionic analysis tools. Underlying our manipulations is a cohomology coefficient associated with an arbitrary monogenic function which we calculate for monogenic constants (i.e., functions in $\ker \overline{\partial} \cap \ker \partial$, where $\overline{\partial}$ and $\partial$ denote, respectively, the generalized Cauchy-Riemann operator and its quaternionic conjugate, see Subsection \ref{Section_Reduced-quaternions} below). Representation formulas and the more important general properties of the elements that constitute the bases are discussed thoroughly.

Proofs of some well-known results are omitted to keep the paper reasonably short.


\section{Toroidal harmonics and monogenics}

\subsection{Basic toroidal harmonics}

The domain on which we wish to study harmonic and monogenic functions is the following. Toroidal coordinates (see \cite{Hob1931,MoonSpencer1988})
for  $x=(x_0,x_1,x_2)\in\R^3$ are given by
 \begin{align} \label{eq:toroidalcoords}
  x_0 &= \frac{\sin \theta}{\cosh\eta-\cos\theta}, \ \
  x_1 = \frac{\sinh \eta \, \cos \varphi}{\cosh\eta-\cos\theta}, \ \  
   x_2 = \frac{\sinh \eta \, \sin \varphi}{\cosh\eta-\cos\theta},
 \end{align}
where $(\eta,\theta,\varphi)\in (0,\infty)\times[-\pi,\pi]\times(-\pi,\pi]$.
We denote the interior toroidal domain as
\begin{align*}
\dom = \{x\colon\ \eta>\eta_0\} \cup \{ \eta=\infty \} \subseteq\R^3,
\end{align*}
where
$\{ \eta=\infty \}$ refers to the limiting circle
$\{x\colon\ x_0=0,\ x_1^2+x_2^2=1\}$. Any torus in $\R^3$ can be
shifted and rescaled to a torus of the form $\dom$.

The associated Legendre functions of the second kind of degree $n$ and order $m$, $Q_n^m(t)$ for $t>1$, are 
required for discussing the harmonic functions in $\dom$. The definition of these functions varies slightly in the literature up to a multiplicative constant
\cite{Bateman1944,Hob1931,WhittakerWatson1927}. For definiteness, we take the representation
 \cite[p.\ 195]{Hob1931}  
\begin{align*} 
  Q_n^m(t) = \frac{(-1)^m}{2^{n + 1}} \frac{(n + m)!}{n!} \, (t^2-1)^{m/2}
  \int_{-1}^1   \frac{(1 - s^2)^n}{(t - s)^{n+m+1}} \,ds
\end{align*}
valid for all complex $n, m$, including half-integers. 

With the abbreviations 
$\Phi\ind{n}{+}(\theta) = \cos n\theta$   and
$\Phi\ind{n}{-}(\theta) = \sin n\theta$,  the  
\textit{interior toroidal harmonic functions} are defined \cite[p.\ 435]{Hob1931}   as
\begin{align} \label{eq:interiorharmonic}
  I\ind{n,m}{\nu,\mu}(x)= 
  \chcos \, Q\ind{n-1/2}{m}(\cosh\eta)  \,
   \Phi\ind{n}{\nu}(\theta) \,  \Phi\ind{m}{\mu}(\varphi)
\end{align}
for integers $n, m$ with $ n\ge0$, $m\ge0$, and signs
$\nu\in\{-1,1\}$, $\mu\in\{-1,1\}$.  It is well known that the
toroidal harmonics are a complete system in the space $\har(\dom)$ of
square integrable real-valued harmonic functions.


\subsection{Reduced quaternions and monogenic functions} \label{Section_Reduced-quaternions}

We denote the quaternions by
\begin{align*}
\H=\{a_0+a_1e_1+a_2e_2+a_3e_3,\ a_i\in\R\},
\end{align*}
where the nonreal units
    $e_1,e_2,e_3$ (often denoted by $i,j,k$) anticommute and satisfy
$e_1e_2e_3=-1$. A \textit{(left) monogenic function} is a
function $f\colon\dom\to\H$ of the form $f=f_0+f_1e_1+f_2e_2+f_3e_3$ such that $\partialbar f=0$, where
\begin{align*}
  \partialbar = \partial_0 + \partial_1e_1 + \partial_2e_2
\end{align*}
with $\partial_i = \partial/\partial x_i$ denotes the usual generalized Cauchy-Riemann (or Fueter) operator. Consult 
\cite{GuerlebeckHabethaSproessig2008,GuerlebeckHabethaSproessig2016,Sudbery1979} for further information 
on monogenic functions.

Let $\M(\H)=\M_{\eta_0}(\H)$ denote the real-linear space of square-integrable $\H$-valued
monogenic functions in $\dom$.  Since the Laplacian in $\R^3$ factors
as $\Delta=\partialbar\partial=\partial\partialbar$ where
\begin{align*}
\partial = \partial_0 - \partial_1e_1 - \partial_2e_2,
\end{align*}
the real-valued components $f_0,f_1,f_2,f_3$ of every
$f\in\M(\H)$ are in $\har(\dom)$. Now, let
\begin{align*}
\A = \{a\in\H\colon\ a_3=0\}
\end{align*}
denote the three-dimensional subspace of \textit{reduced quaternions}, and $\M(\A)$ the subset of $\M(\H)$ formed of
$\A$-valued monogenic functions.  Then $\partial h\in\M(\A)$ when
$h\in\har(\dom)$.
 
To put our subject in context, it is helpful to bear in mind the
standard construction of monogenic functions defined in a ball
\cite{Cac2004,CGBHanoi2005,Cacao2010,MoraisGurlebeck2012} by applying $\partial$
to the solid spherical harmonics, represented in spherical coordinates
by $U\ind{n,m}{\pm} = r^nP_n^m(\cos\theta)\Phi\ind{m}{\pm}(\varphi)$, where $P_n^m$ denotes the associated Legendre function of the first kind. It 
is proved that the monogenic functions
$X\ind{n,m}{\pm}=\partial U\ind{n+1,m}{\pm}$ thus obtained form a
complete set in the space of $\A$-valued $L^2$ monogenic functions in
the three-ball; in fact, they are an orthogonal Hilbert-space
basis. These results have been extended to monogenic functions defined
in a spheroid in
\cite{Morais2011,Morais2013,Morais2019,MoraisHabilitation2021}.

After a closer examination of the
behavior of toroidal harmonics, we will see, however (see Theorem
\ref{theo:basisAvalued} below), that the structure of the monogenic
functions on a torus is very different: application of the operator $\partial$ to
interior toroidal harmonics falls far short of giving a complete set in
$\M(\A)$. One reason is since the origin in $\R^3$ is not in $\dom$,
functions analogous to $\re z^n$ and $\im z^n$ are not eliminated for
negative $n$; the fact that $\dom$ is not simply connected adds a
further complication.  However, the most surprising difference results
from the fact that $\partial_0$ shifts the index from $n$ to $n+1$ rather
than $n-1$, as in the case of the ball (see Section \ref{sec:appell} below).
 
To take into account the topology of the torus, we associate to an $\A$-valued function $f = f_0+f_1e_1+f_2e_2$ the real differential 1-form
\begin{align}  \label{eq:omegaf}
  \omega_f = f_0\,d x_0 - f_1\,d x_1-  f_2\,d x_2  
\end{align}
which may be thought of as the scalar part of  $f(x)(d x_0e_0+d x_1e_1+d x_2e_2)$. It is easily verified that $\omega_f$ is a closed form 
precisely when $\overline{f}$ is irrotational (i.e., ${\rm curl}\ \overline{f} = 0$). In particular, $\omega_f$ is a closed form when $f$ is an $\A$-valued monogenic
function. When $f$ is constructed as $f=\partial h$ for a real-valued harmonic function $h$, the form $\omega_f=dh$ is exact.
This gives the following.

\begin{proposition} If $f$ is an $\A$-valued monogenic function, then $\omega_f$ is a closed form.
  If further $f$ is defined in a simply-connected domain in $\R^3$, the line integral
  $\int_a^b\omega_f$ does not depend on the curve from $a$ to $b$ in
  that domain.
\end{proposition}

\begin{definition}
Let $f \in \M(\A)$ be a monogenic function in $\dom$. We define the \textit{cohomology coefficient} of $f$ in $\dom$ by integration along the unit circle
  in the plane $x_0=0$:
  \begin{align}   \label{eq:defcoh}
    \coh f = \frac{1}{2\pi}\int_{\eta=\infty} \, d\omega_f .
  \end{align}
\end{definition}
The orientation of $\{\eta=\infty\}$ is given by the parametrization $(0,\cos t, \sin t)$.
From this definition, $\coh f=0$ when $f$ is an \textit{exact monogenic function}, i.e., $f = \partial h$. (The 
converse is also true, but we will not need this fact.)


\subsection{$\A$-valued monogenic constants}

In this section we will continue to assume that $f_3=0$.
A \textit{monogenic constant} is a function $f$ for which $\partialbar f =\partial f =0$.  Write $\MC(\A)$ for the 
subspace of $\A$-valued monogenic constants on $\dom$. Since these functions satisfy that 
$\partial_0f=0$ and $ \partial_1fe_1 + \partial_2fe_2=0$,
the following well-known fact is immediate.

\begin{proposition} \label{prop:monogconst}
  Let $f\colon\dom\to\A$ be a monogenic constant. Then $f_0\in\R$ is a constant, while
  $f_1,f_2$ are functions of $x_1,x_2$ and $f_1-if_2$ is a holomorphic function
  of the complex variable $x_1+ix_2$. 
\end{proposition}


Further, if $f\in\M(\A)$ and the scalar part $f_0$ of $f$ is a constant, then $f\in\MC(\A)$. Proposition \ref{prop:monogconst} and the geometry of the torus lead us to consider the following.
\begin{definition}
Let $m\in\Z$. For $x\in\R^3$ define the scalar-valued functions
\begin{align}\label{eq:defJ}
   J_m^+(x) = \re (x_1+ix_2)^m, \quad  J_m^-(x) = \im(x_1+ix_2)^m .
\end{align}
\end{definition}

These functions are well-defined in $\dom$ and independent of $x_0$. Clearly $\pm J_m^\mp$ is a harmonic conjugate of $J_m^\pm$ when considered
as functions of the complex variable $x_1+ix_2$. Also let
\begin{align} \label{eq:Jhat}
    \widehat J(x) = -\log|x_1+ix_2|,
\end{align}
which does not admit a harmonic conjugate since $\arg|x_1+ix_2|$ is
not single-valued. It is easily verified by the use of Laurent series that
any $u\in\har(\dom)$, which is independent of $x_0$, has a unique
representation as
\begin{align} \label{eq:uniqueharmonicseries}
  u = a\widehat J + \sum_{m=-\infty}^\infty( a_m^+ J_m^+ +a_m^- J_m^-)
\end{align}
with real coefficients $a,a_m^\pm$. In this and all similar sums,
$J_0^-$ is to be excluded since it is identically zero. It also follows that
\begin{align*}
  \partial_0 J_m^\pm=0, \quad \partial_1 J_m^\pm = mJ_{m-1}^\pm, \quad
  \partial_2 J_m^\pm = \mp mJ_{m-1}^\mp.
\end{align*}

\begin{definition}
We introduce the \textit{basic monogenic constants} on $\dom$,
\begin{align}  \label{eq:defW}
  W_m^\pm = J_m^\pm e_1 \mp J_m^\mp e_2
\end{align}
for all integer $m$.
\end{definition}

We have
\begin{align}  
  W_m^\pm &= \frac{-1}{m+1}\,\partial J_{m+1}^\pm, \quad
  m\not=-1,  \nonumber\\
  W_{-1}^+ & = \partial\widehat J. \label{eq:appellW}
\end{align}
 However, $W_{-1}^-$ is not obtained by applying $\partial$
 to a real-valued harmonic function, although locally it is equal
 to the $\A$-valued function $-\partial \arctan(x_2/x_1)$. In particular,
 \begin{align} \label{eq:cohw-1}
   \coh W_{-1}^-=1    
 \end{align}
 while the
 other basic monogenic constants $W_m^\pm$ have vanishing cohomology coefficient.

\begin{proposition}
For all $m\in\Z$, we have  the following representation in terms of the toroidal harmonics:
\begin{align} \label{eq:JfromI}  
  J_m^\pm =  \sum_{n=0}^\infty  j_{n,m} \, I\ind{n,|m|}{+,\pm}   
\end{align}
where
 \begin{align*}
   j_{n,m}  = \left\{
   \begin{array}{ll}
    \displaystyle (1+\delta_{0,m}) (-1)^m \sqrt{\frac{2}{\pi}}\frac{1}{\Gamma(m+\frac{1}{2})}, & m\ge0, \\[2ex]
      \displaystyle \pm  2 (-1)^m \sqrt{\frac{2}{\pi}}\frac{1}{\Gamma(m+\frac{1}{2})}\frac{\Gamma(n+m+1/2)}{\Gamma(n-m+1/2)}, \ & m<0.
   \end{array} \right.
 \end{align*}
We use the Kronecker symbol: $\delta_{m_1,m_2} = 0$ or $1$, according as $m_1 \neq m_2$, or $m_1 = m_2$.
\end{proposition}

\begin{proof}
A Fourier series for powers of $\cosh\eta-\cos\theta$
  with constant $\eta$ was given in \cite[Eq.\ (3.10)]{CD2011},
  valid for all $\alpha\in\C$:
\begin{align} \label{eq:cohldominici}
  (\cosh\eta -\cos\theta)^{-\alpha} = \frac{1}{\Gamma(\alpha)}\sqrt{\frac{2}{\pi}}
    \frac{e^{-i\pi(\alpha-1/2)}}{(\sinh\eta)^{\alpha-1/2}}
    \sum_{n=0}^\infty  (1+\delta_{0,n}) \,    Q_{n-1/2}^{\alpha-1/2}(\cosh\eta) \cos(n\theta) .
\end{align}
By \eqref{eq:toroidalcoords},
\[ (x_1+ix_2)^m =
  \frac{\sinh^m\eta}{(\cosh\eta-\cos\theta)^m} \, e^{im\varphi}, \] and
when we look at the series for
$J_m^\pm(x)/\sqrt{\cosh\eta-\cos\theta}$, i.e., with $\alpha=m+1/2$,
the values of coefficients $j_{n,m}$ become apparent.
\end{proof}

By \eqref{eq:defW} and \eqref{eq:JfromI}, we can express the basic
monogenic constants as
\begin{align} \label{eq:WfromI}
  W_m^\pm = \sum_{n=0}^\infty j_{n,m}(I\ind{n,m}{+,\pm}e_1 \mp I\ind{n,m}{+,\mp}e_2)
\end{align}
for all $m \in \Z$, although the individual summands are generally not monogenic.
It can further be shown that

\begin{proposition} \label{prop:monogconstrep}
  Every  $\varphi\in\MC(\A)$ is of the form
  \begin{align*}
    \varphi = a_0 + \sum_{m=-\infty}^\infty \
    (a_m^+ W_m^+ + a_m^- W_m^-)  ,
  \end{align*}
where $a_0,a^\pm_m \in \R$, converging uniformly on compact subsets.
\end{proposition}

\begin{proof}
  Write $\varphi=a_0+\varphi_1e_1+\varphi_2e_2$ according to Proposition
  \ref{prop:monogconst}.  Express
  $\varphi_1=\hat a\widehat J + \sum_{-\infty}^\infty (a_m^+ J_m^++a_m^- J_m^-)$
  according to \eqref{eq:uniqueharmonicseries}. Since $\varphi_1$ admits
  $-\varphi_2$ as a harmonic conjugate and the sum
  $ \sum_{-\infty}^\infty (a_m^+ J_m^++a_m^- J_m^-)$ admits
  $\sum (a_m^+ J_m^- - a_m^- J_m^+) $ as a harmonic conjugate, it
  follows that $\hat a\widehat J$ also admits a harmonic conjugate, so
  $\hat a=0$. It also follows that  $\varphi_2$ differs from $\sum (a_m^+ J_m^- - a_m^- J_m^+)$
  by an additive real constant $a_0$. 
\end{proof}


\section{Appell basis for toroidal harmonics \label{sec:appell}}

The solid spherical harmonics $U\ind{n,m}{\pm}$  satisfy \cite{MoraisGurlebeck2012} a well-known ``Appell'' property with respect to
the partial derivative $\partial_0$, that is,
$\partial_0U\ind{n,m}{\pm}$ is a constant multiple of
$U\ind{n-1,m}{\pm}$. Similarly, the spherical monogenic functions
$X\ind{n,m}{\pm}$ satisfy the Appell-type property that
$\partial X\ind{n,m}{\pm}$ is a constant multiple of
$X\ind{n-1,m}{\pm}$ \cite{Cac2004}. However, the toroidal harmonics behave somewhat differently:  
\begin{proposition} \label{prop:partialI0}
  Let $n,m\ge1$. Then for all combinations $\nu,\mu=\pm1$,
  \begin{align*}
    \partial_0 I\ind{0,0}{\nu,\mu} &= -\frac{1}{2}I\ind{1,0}{-\nu,\mu}, \\
    \partial_0 I\ind{n,0}{\nu,\mu} &= (\nu)\big(-\frac{2n-1}{4}I\ind{n-1,0}{-\nu,\mu} 
    + n I\ind{n,0}{-\nu,\mu}  - \frac{2n+1}{4}  I\ind{n+1,0}{-\nu,\mu} \big)
                                    ,  \\
    \partial_0 I\ind{0,m}{\nu,\mu} &=  \frac{2m-1}{2}I\ind{1,m}{-\nu,\mu}
                                      ,  \\
    \partial_0 I\ind{n,m}{\nu,\mu} &= (\nu) \big(
       -\frac{2n+2m-1}{4}I\ind{n-1,m}{-\nu,\mu} 
       + n I\ind{n,m}{-\nu,\mu} - \frac{2n+2m+1}{4} I\ind{n+1,m}{-\nu,\mu} \big).
  \end{align*}
\end{proposition}

These expressions are the scalar parts of the toroidal monogenics $T\ind{n,m}{\nu,\mu}$ (see Definition \ref{defi:monogbasics} below).
For completeness, we provide the $e_1$ and $e_2$ components of
$T\ind{n,m}{\nu,\mu}$ as well, although we will not use the explicit formulas in the
sequel.
\begin{proposition}  \label{prop:partialI12}
For all combinations $\nu,\mu=\pm1$, we have
 \begin{align*}
    \partial_1 I\ind{0,0}{\nu,\mu} &= I\ind{0,1}{\nu,\mu} - I\ind{1,1}{\nu,\mu}, \\
    \partial_1 I\ind{n,0}{\nu,\mu} &=  -\frac{1}{2}I\ind{n-1,1}{\nu,\mu} 
    + I\ind{n,1}{\nu,\mu} -\frac{1}{2} I\ind{n+1,1}{\nu,\mu}  
                                     \quad (n\ge1),  \\
    \partial_1 I\ind{0,m}{\nu,\mu} &= (\nu)\big( -\frac{(2m-3)(2m-1)}{8}I\ind{1,m-1}{\nu,\mu} 
     -\frac{1}{2} I\ind{1,m+1}{\nu,\mu} - \frac{(2m-1)(2m+1)}{8} I\ind{0,m-1}{\nu,\mu}\\
   & \quad\  +\frac{1}{2} I\ind{0,m+1}{\nu,\mu} \big)           \quad (n\ge1),  \\
    \partial_1 I\ind{n,m}{\nu,\mu} &=  
        -\frac{(2n+2m-3)(2n+2m-1)}{16}I\ind{n-1,m-1}{\nu,\mu} 
        -\frac{1}{4} I\ind{n-1,m+1}{\nu,\mu}, \\
     & \quad\ +\frac{(2n+2m-1)(2n-2m+1)}{16} I\ind{n,m-1}{\nu,\mu}
        +\frac{1}{2} I\ind{n,m+1}{\nu,\mu} , \\     
     & \quad\ -\frac{(2n-2m+1)(2n-2m+3)}{16} I\ind{n+1,m-1}{\nu,\mu}
        -\frac{1}{4} I\ind{n+1,m+1}{\nu,\mu}       
                                     \quad (n,m\ge1), 
 \end{align*}
 \begin{align*}
    \partial_2 I\ind{0,0}{\nu,\mu} &= I\ind{0,1}{\nu,-\mu} - I\ind{1,1}{\nu,-\mu}, \\
    \partial_2 I\ind{n,0}{\nu,\mu} &= (\mu)\big( -\frac{1}{2}I\ind{n-1,1}{\nu,-\mu} 
    + I\ind{n,1}{\nu,-\mu} - \frac{1}{2} I\ind{n+1,1}{\nu,-\mu}   \big)
                                     \quad (n\ge1),  \\
    \partial_2 I\ind{0,m}{\nu,\mu} &= (\mu)\big( \frac{(2m-3)(2m-1)}{8}I\ind{1,m-1}{\nu,-\mu} 
     -\frac{1}{2} I\ind{1,m+1}{\nu,-\mu} + \frac{(2m-1)(2m+1)}{8} I\ind{0,m-1}{\nu,-\mu}
       \\
   & \quad\  +\frac{1}{2} I\ind{0,m+1}{\nu,-\mu} \big)          \quad (n\ge1),  \\
    \partial_2 I\ind{n,m}{\nu,\mu} &=  (\mu)\big(
        \frac{(2n+2m-3)(2n+2m-1)}{16}I\ind{n-1,m-1}{\nu,-\mu} 
        -\frac{1}{4} I\ind{n-1,m+1}{\nu,-\mu}, \\
     & \quad\ -\frac{(2n+2m-1)(2n-2m+1)}{16} I\ind{n,m-1}{\nu,-\mu}
        +\frac{1}{2} I\ind{n,m+1}{\nu,-\mu} , \\     
     & \quad\ +\frac{(2n-2m+1)(2n-2m+3)}{16} I\ind{n+1,m-1}{\nu,-\mu}
        -\frac{1}{4} I\ind{n+1,m+1}{\nu,-\mu}       
                                 \big)    \quad (n,m\ge1).
  \end{align*}
\end{proposition}

The  formulas in Proposition \ref{prop:partialI0} may be expressed succinctly as
\begin{align}  \label{eq:partialI}
  \partial_0 I\ind{n,m}{\nu,\mu} =
  \sum_{k=0}^{\infty}\kappa_{k,m}^{n} I\ind{k,m}{-\nu,\mu} =
  \sum_{k=(n-1)_+}^{n+1}\kappa_{k,m}^{n}I\ind{k,m}{-\nu,\mu},
\end{align}
where we denote $(t)_+=\max(t,0)$, with the symbols
\begin{align*}
    \kappa_{0,m}^{0} &= m-\frac{1}{2} \quad (m\ge1)  ,  
\end{align*}
while for  $n\ge1$, 
\begin{align*}
 \kappa_{k,m}^{n} &= 
  \begin{cases}
   -\frac{1}{2}(n+m-\frac{1}{2}), & \ \ \ k=n-1, \\
   \ n,& \ \ \ k=n,\\
    -\frac{1}{2}(n-m+\frac{1}{2}), & \ \ \ k=n+1,\\
  \end{cases}
\end{align*}
with $\kappa_{k,m}^{n}=0$ otherwise.\\

\begin{proof}[Proof of  Proposition \ref{prop:partialI0}]
  For simplicity of presentation, we only address the calculations for $(\partial/\partial x_0) I\ind{n,0}{+,+}$. The calculations for the remaining cases follow the same principle and are therefore omitted.

In view of the definition \eqref{eq:interiorharmonic}, direct computations show that
    \begin{align*} 
 \frac{\partial}{\partial x_{0}}I\ind{n,0}{+,+}&=\frac{\partial}{\partial x_{0}}\bigg((\cosh\eta-\cos\theta)^{1/2}Q_{n-\frac{1}{2}}^0(\cosh\eta)\cos(n\theta)\bigg)\nonumber  \\
  &=-  \sin\theta\sinh\eta\bigg(\frac{\sinh\eta}{2(\cosh\eta-\cos\theta)^{1/2}}\,Q_{n-\frac{1}{2}}^0(\cosh\eta)\cos(n\theta)\nonumber\\    
  & \qquad \qquad   +\sinh\eta \,(Q_{n-\frac{1}{2}}^{0})'(\cosh\eta)(\cosh\eta-\cos\theta)^{1/2}\cos(n\theta)\bigg)\nonumber\\
  &\quad+ (\cosh\eta\cos\theta-1)\bigg(\frac{\sin\theta}{2(\cosh\eta-\cos\theta)^{1/2}}\,Q_{n-\frac{1}{2}}^0(\cosh\eta)\cos(n\theta)\nonumber\\ & \qquad    \qquad     - \, n\,(\cosh\eta-\cos\theta)^{1/2}\,Q_{n-\frac{1}{2}}^0(\cosh\eta)\sin(n\theta)\bigg) \nonumber\\
  &= (\cosh\eta-\cos\theta)^{1/2}\bigg(-\frac{1}{2}\,Q_{n-\frac{1}{2}}^0(\cosh\eta)\cosh\eta\cos(n\theta)\sin\theta\nonumber\\ & \qquad  \qquad \qquad \qquad  \qquad -(\cosh\eta\cos\theta-1)\,Q_{n-\frac{1}{2}}^0(\cosh\eta)\sin(n\theta)\nonumber\\& \qquad  \qquad \qquad \qquad \qquad -\sinh^{2}\eta\,(Q_{n-\frac{1}{2}}^{0})'(\cosh\eta)\cos(n\theta)\sin\theta \bigg).
\end{align*}
Applying the recursion formula
\begin{align*}
(1-t^{2})(Q_{n+1}^{m})^{\prime}(t)=(n+m+1)Q_{n}^{m}(t)-(n+1)t Q_{n+1}^{m}(t),
\end{align*}
and some trigonometric identities, we obtain
\begin{align*} 
       (\cosh\eta-\cos\theta)^{1/2}\bigg(&-n\,\cosh\eta \,Q_{n-\frac{1}{2}}^0(\cosh\eta)\cos(n\theta)\sin\theta\nonumber\\&
       -n\,(\cosh\eta\cos\theta-1)\,Q_{n-\frac{1}{2}}^0(\cosh\eta)\sin(n\theta)\nonumber\\&
       +(n-\frac{1}{2})\,Q_{n-\frac{3}{2}}^0(\cosh\eta)\cos(n\theta)\sin\theta\bigg).
\end{align*}
Now use the recursion formula
\begin{align*}
    (n-m+1)Q_{n+1}^{m}(t)=(2n+1)t Q_{n}^{m}(t)-(n+m)Q_{n-1}^{m}(t),
\end{align*}
with $n-1/2$ in place of $n$ and with $m=0$ to convert the expression into
\begin{align*}
   (\cosh\eta-\cos\theta)^{1/2}\bigg(  & -\frac{2n+1}{4}\,Q_{n+\frac{1}{2}}^0(\cosh\eta)\sin (n\theta)+\,n\,Q_{n-\frac{1}{2}}^0(\cosh\eta)\sin(n\theta)\nonumber \\
        &  -\frac{2n-1}{4}\,Q_{n-\frac{3}{2}}^0(\cosh\eta)\sin (n-1)\theta\bigg).
\end{align*}
This establishes the statement.  
\end{proof}

The proof of Proposition \ref{prop:partialI12} is similar, although more
complicated, and will be omitted.

The foregoing results show that the toroidal harmonics do not form an
Appell system with respect to any of the three partial derivatives
$\partial/\partial x_i$. To obtain harmonics satisfying an Appell-type
property with respect to $\partial/\partial x_0$, we introduce a
change of basis for the toroidal harmonics as follows: Set
$i\ind{n,m}{*\,n}=1$, and having defined $i\ind{0,m}{*\,n-1},\dots,i\ind{n-1,m}{*\,n-1} $, then for $0\le k\le n-1$ let
\begin{align}
i\ind{k,m}{*\,n}=\frac{1}{\kappa_{n,m}^{n-1}}\sum_{j=(k-1)_+}^{n-1}\kappa_{k,m}^{j}i\ind{j,m}{*\,n-1}.
\end{align}
Further, set $i\ind{n,m}{n}=1$ and then recursively for $k=n-1,n-2,\dots,0$ let
\begin{align} \label{definitionitoistar}
  i\ind{k,m}{n}=\,-\sum _{j=k+1}^{n} i\ind{k,m}{*\,j}\, i\ind{j,m}{n}.
\end{align}
Finally, we define a new set of harmonic functions by 
\begin{align} \label{eq:Istar}
    I^{*\,\nu,\mu}_{n,m}=\sum_{k=0}^{n} i\ind{k,m}{*\,n}I\ind{k,m}{\nu,\mu}.    
\end{align}

From this construction, we have
\begin{proposition} \label{prop:appellI} The collection
  $\{I\ind{n,m}{*\,\nu,\mu}\}$ satisfies the following reverse
  Appell-type property:
\begin{align} \label{eq:reverseAppell}
\frac{\partial I\ind{n,m}{*\,\nu,\mu}}{\partial x_{0}} =\kappa_{n+1,m}^{n}I\ind{n+1,m}{*\,-\nu,\mu} \quad(n\ge0).
\end{align}
Also, we have the inverse relation
\begin{align} \label{eq:IfromIstar}
   I\ind{n,m}{\nu,\mu} = \sum _{k=0}^{n}i\ind{k,m}{n}I\ind{k,m}{*\,\nu,\mu}.
\end{align}
\end{proposition}

\begin{proof}
To prove \eqref{eq:reverseAppell}, first we note that
  \begin{align*}
    \frac{\partial I\ind{n,m}{*\, \nu,\mu}}{\partial x_{0}}&= \frac{\partial}{\partial x_{0}}\sum_{j=0}^{n}i^{*\,n}_{j,m} I\ind{j,m}{\nu,\mu} \\
                                                         &=\sum_{j=0}^{n}i^{*\,n}_{j,m}\sum_{k=(j-1)_+}^{j+1}\!\!\!\kappa_{k,m}^{j}I\ind{k,m}{-\nu,\mu} \\
                                                         &=\sum_{k=0}^{n+1}\bigg(\sum_{j=(k-1)_+}^{n}\!\!\!\kappa_{k,m}^{j}i\ind{j,m}{*\,n}\bigg)I\ind{k,m}{-\nu,\mu}\\ &
                                                                                                                                                                    =\sum_{k=0}^{n}\bigg(\sum_{j=(k-1)_+}^{n}\!\!\!\kappa_{k,m}^{j}i\ind{j,m}{*\,n}\bigg)I\ind{k,m}{-\nu,\mu} \ + \ \kappa_{n+1,m}^{n}i\ind{n+1,m}{*\,n+1}I\ind{n+1,m}{-\nu,\mu}.
  \end{align*}
  Since $\kappa_{n+1,m}^{n}i\ind{k,m}{*\,n+1}=\sum_{j=(k-1)_+}^{n}\kappa_{k,m}^{j}i\ind{j,m}{*\,n}$ and $i\ind{n+1,m}{*\,n+1}=1$ we have 
  \begin{align*}
    \frac{\partial I\ind{n,m}{*\,\nu,\mu}}{\partial x_{0}}
    =&\,\sum_{k=0}^{n}\kappa_{n+1,m}^{n}i\ind{k,m}{*\,n+1}I\ind{k,m}{-\nu,\mu}+\kappa_{n+1,m}^{n}I\ind{n+1,m}{-\nu,\mu}\\=&\kappa_{n+1,m}^{n}\sum_{k=0}^{n+1}i\ind{k,m}{*\,n+1}I\ind{k,m}{-\nu,\mu}\\
    =&\,\kappa_{n+1,m}^{n}I\ind{n+1,m}{*\, -\nu,\mu} ,
  \end{align*}
  which establishes \eqref{eq:reverseAppell}.
  
To verify \eqref{eq:IfromIstar}, simply note that by \eqref{eq:Istar} the vector $(I\ind{n,m}{*\,\nu,\mu})_n$ is the image of the vector $(I\ind{k,m}{\nu,\mu})_k$ under the triangular matrix $(i\ind{k,m}{*\,j})_{k,j}$, which by 
\eqref{definitionitoistar} is the inverse of the matrix $(i\ind{j,m}{n})_{j,n}$.
\end{proof}

We will use $\Spanbar$ to indicate the closure of the linear span in $L^2(\dom)$, or in $L^2$, according to context. It is well known that convergent series of analytic functions in $L^2$ also converge uniformly on compact subsets, and thus are analytic; subspaces of harmonic or monogenic functions are closed in $L^2$.

It follows from \eqref{eq:IfromIstar} that
$\Spanbar\{ I\ind{n,m}{\nu,\mu}\colon\ n\le n_0\}= \Spanbar\{
I\ind{n,m}{*\,\nu,\mu}\colon\ n\le n_0\}$ for every $n_0$, and that
the collection $\{I\ind{n,m}{*\,\nu,\mu}\}$ is also a complete set for
$\har(\dom)$. This permits us to prove the following preliminary step
in our understanding of toroidal harmonics.

 \begin{proposition} \label{prop:primitiveharmonic0}
   Let $f_0\in\har(\dom)$. Then there exists $h\in\har(\dom)$ such that
 $f_0=\partial_0h$ if and only if $f_0\in\Spanbar\{I\ind{n,m}{\nu,\mu}\colon\ n\ge1\}$. 
\end{proposition}

\begin{proof}
  By \eqref{eq:Istar}, we have $I\ind{0,m}{*\,\nu,\mu}=I\ind{0,m}{\nu,\mu}$. Supposing the condition,
  we may write $f_0 = \sum_{n\ge1}\sum_{m,\nu,\mu} a\ind{n,m}{\nu,\mu}I\ind{n,m}{*\,\nu,\mu}$. Then by Proposition \ref{prop:appellI},
  $f_0= \partial_0 h$ where
 \[h=\sum_{n\ge1}\sum_{m,\nu,\mu}  a\ind{n,m}{\nu,\mu}(1/\kappa_{n,m}^{n-1})
  I\ind{n-1,m}{*\,-\nu,\mu}. \]

  Now suppose that $f_0= \partial_0 h$. Write
  $h=\sum_{n\ge0} b\ind{n,m}{\nu,\mu}I\ind{n,m}{*\,\nu,\mu}$ and again apply Proposition
  \ref{prop:appellI} to see that $f_0\in\Spanbar\{I\ind{n,m}{\nu,\mu}\colon\ n\ge1\}$. This concludes the proof.
\end{proof}


\section{Basis for $\A$-valued monogenic functions} \label{Section_Basis_toroidal_monogenics1}

The reverse Appell-type property with respect to $\partial_0$ of
Proposition \ref{prop:appellI} motivates the following definition.

\begin{definition} \label{defi:monogbasics}
The \textit{exact basic toroidal monogenic functions} are
  \begin{align}  \label{eq:monogbasics}
    T\ind{n,m}{\nu,\mu} = \partial  I\ind{n-1,m}{*\,-\nu,\mu},\ n\ge 1 .
  \end{align}
\end{definition}

Clearly $T\ind{n,m}{\nu,\mu}\in\M(\A)$ and $\coh T\ind{n,m}{\nu,\mu}=0$. Since every harmonic function
is in the closure of the real-linear span of
$\{I\ind{n,m}{\nu,\mu}\}$, the following is trivial. (It is not difficult to verify that the interchange of summation and differentiation occurring in Proposition \ref{Proposition_span} below is permitted.)

\begin{proposition} \label{Proposition_span}
  Let $h\in\har(\dom)$. Then  $\partial h\in \Spanbar\{T\ind{n,m}{\nu,\mu}\colon\ n\ge1\}$.
\end{proposition}

Proposition \ref{prop:primitiveharmonic0} tells us immediately
that the harmonics $I\ind{n,m}{*\,\nu,\mu}$ and  $I\ind{n,m}{\nu,\mu}$ are scalar parts of monogenic functions when $n\ge1$.  The
verification for $n=0$ follows along different lines. 

\def\teo{\operatorname{\mathcal{T}\!}}
The \textit{Teodorescu operator} for a bounded domain $D\subset\C$
is given by
\[  \teo_Df(w) = \frac{-1}{\pi}\int_D \frac{f(z)}{z-w} \,dx\,dy.
\]
It satisfies  
 $(\partial \teo_Df(z))/(\partial\overline{z}) = f(z)$ (see \cite{AF2003,Ahl2006}).  We take $D$ to be the annulus with
inner and outer radii $\sinh\eta_0/(\cosh\eta_0\pm1)$, which
corresponds to the slice of $\dom$ at $x_0=0$.

Consider the following standard construction \cite{Sudbery1979} starting
with a real-valued function $f_0$ in $\dom$. Let
$w(z)=\teo_D((\partial_0f_0)(0,x,y))= w_1(z)+iw_2(z)$, $z = x + iy \in \C$, so
\[ \frac{1}{2}\big( \frac{\partial w_1}{\partial x}-
  \frac{\partial w_2}{\partial y} \big) = \partial_0f_0, \quad
  \frac{1}{2}\big( \frac{\partial w_2}{\partial x}+
  \frac{\partial w_1}{\partial y} \big) = 0.
\]
Define
\begin{align*}
\vec v(x_1,x_2) =    -\frac{1}{2}w_1(x_1+ix_2)e_1 +\frac{1}{2}w_2(x_1+ix_2)e_2,
\end{align*}
so $(\partial_1e_1+\partial_2e_2)\vec v = \partial_0f_0$.
Next, let
\[ f_1(x)e_1+ f_2(x)e_2 = - \int_0^{x_0}
  (\partial_1e_1+\partial_2e_2) f_0(t,x_1,x_2)\,dt -  \vec v(x_1,x_2) \]
which is well defined for $x\in\dom$.

This construction defines the operator
\begin{align}   \label{eq:defPsi}
  \Psi[f_0] = f_0+ f_1e_1+ f_2e_2 .
\end{align}
One may verify that when $f_0$ is harmonic, $\Psi[f_0]$ is
monogenic. We complement \eqref{eq:monogbasics} as follows.

\begin{definition} The
\textit{basic toroidal harmonics for index $n=0$} are
  \begin{align} \label{eq:T0m}
    T\ind{0,m}{+,\mu} = \Psi[  I\ind{0,m}{+,\mu} ] - (\coh \Psi[  I\ind{0,m}{+,\mu} ])W_{-1}^-.
  \end{align}
\end{definition}

From the fact that $I\ind{0,m}{+,\mu}$ and $\partial_i I\ind{0,m}{+,\mu}$ are bounded in $\overline{\Omega}_{\eta_0}$ (in fact, in a larger domain $\Omega_{\eta_0-\epsilon}$), it follows that $w_1, w_2$ as well as $\vec v$ in the above construction are also bounded. Since 
$W_{-1}^- \in L^2(\dom)$, we have $T\ind{0,m}{+,\mu}\in\M(\A)$. Further, by \eqref{eq:cohw-1}, $\coh T\ind{0,m}{+,\mu}=0$.
\begin{lemma} \label{lemm:Tindependent}
  The monogenic functions $T\ind{n,m}{\nu,\mu}$ defined in \eqref{eq:monogbasics}, \eqref{eq:T0m} are linearly independent over $\R$.
\end{lemma}
 
\begin{proof}
This is equivalent to showing that the monogenic functions
  \begin{align*}
    T\ind{0,m}{+,\mu} \mbox{\, and \,}
    T\ind{n,m}{*\,\nu,\mu} = \sum_{k=1}^{n} i\ind{k,m}{*\,n}T\ind{k,m}{\nu,\mu}  \ (n\ge 1)
  \end{align*}
  are linearly independent, which follows from the fact that the
 collection of  scalar parts $I\ind{0,m}{+,\pm}$ and
  $(1/\kappa_{n,m}^{n-1})I\ind{n-1,m}{*\,\pm,\pm}$ ($n\ge1$) is
  linearly independent.
\end{proof}

We now have the necessary material in hand for giving a complete
independent set in $\M(\A)$. From \eqref{eq:JfromI}, \eqref{eq:Istar}, and Proposition \ref{prop:appellI}, we have the following relationships among linear
spans:
\begin{align*}
  \Spanbar\{J_m^\pm\} \subseteq \Spanbar\{I\ind{n,m}{\nu,\mu}\} =
   \Spanbar\{I\ind{n,m}{*\,\nu,\mu}\} = \har(\dom).
\end{align*}

\begin{theorem} \label{theo:basisAvalued}
  Every $f\in\M(\A)$ has a unique representation
\begin{align}  \label{eq:basisA}
  f =    \sum_{n,m,\nu,\mu} a\ind{n,m}{\nu,\mu} T\ind{n,m}{\nu,\mu}
  +\sum_{m=-\infty}^\infty b\ind{n,m}{\nu,\mu}\sum_{\mu=\pm1} W_m^\mu
\end{align}
with real coefficients $a\ind{n,m}{\nu,\mu}$,  $b\ind{n,m}{\nu,\mu}$.
\end{theorem}

\begin{proof}
 Write $f=f_0+f_1e_1+f_2e_2$, and decompose
  $f_0=h_0+h_1$ with $h_0\in\Spanbar\{I\ind{0,m}{\nu,\mu}\}$ and $h_1\in\Spanbar\{I\ind{n,m}{\nu,\mu}\colon\ n\ge1\}$.
  By Proposition \ref{prop:primitiveharmonic0}, take $h$ so that
  $\partial_0h=h_1$. The function 
  \begin{align*}  
    \varphi= f - (\Psi[h_0] + \partial h)
  \end{align*}
  satisfies $\Sc \varphi=0$, so $\varphi\in\MC(\A)$. By the property
  characterizing $h_0$ we can take
  $g\in\Span(\{T\ind{0,m}{+,\mu}\})$ with $\Sc g=h_0$. Further,
  take $h$ such that $\partial_0h=h_1$. Now let
  \[  \psi = \Psi[h_0] - g, \]
  which is also in $\MC(\A)$. Thus we have a decomposition
  \[   f = (g+\partial h) + (\varphi+\psi), \]
  where the first term is in the span of the $T\ind{n,m}{\nu,\mu}$
  and the second in the span of the $W_m^\mu$.
  
  For the uniqueness, we may suppose that $f=0$ in  \eqref{eq:basisA}. By \eqref{eq:reverseAppell}, \eqref{eq:monogbasics}, 
  and \eqref{eq:T0m}, the scalar part of $f$ is
  \begin{align*}
    0 = \Sc f = \sum a\ind{n,m}{\nu,\mu} \Sc T\ind{n,m}{\nu,\mu} 
    = \sum \lambda_{n,m} \, a\ind{n,m}{\nu,\mu} I\ind{n,m}{*\,\nu,\mu},
  \end{align*}
where $\lambda_{0,m} = 1$ and $\lambda_{n,m} = \kappa_{n,m}^{n-1}$ for $n \geq 1$. Therefore $a\ind{n,m}{\nu,\mu}=0$ for all indices, leaving $\sum b\ind{n,m}{\nu,\mu}\sum_{\mu=\pm1} W_m^\mu =0$, so $b\ind{n,m}{\nu,\mu} = 0$. Therefore  \eqref{eq:basisA} is independent.  
\end{proof}

While it would be very interesting to find an orthogonal basis
in $\M(\A)$, such a result would be beyond the scope of this paper.


\section{Basis for $\H$-valued monogenic functions} \label{Section_Basis_toroidal_monogenics2}

We now consider functions in $\M(\H)$; the domain of all function spaces mentioned continues to be $\dom$. One way to construct such
functions in terms of elements of $\M(\A)$ is with the aid of the following two results.
 
\begin{lemma} \label{lemm:semiunique}
  Let $f,g\in\M(\A)$ be such that $f+ge_3=0$ identically. Then
  $f,g\in\MC(\A)$.
\end{lemma}

\begin{proof}  Since
\[ f_0 + (f_1+g_2)e_1 + (f_2-g_1)e_2 + g_0 e_3 = 0, \]
we have $f_0=g_0=0$, so $f$ and $g$ are monogenic constants.
\end{proof}

Note that when $\varphi\in\MC(\A)$, then also $\varphi e_3\in \MC(\A)$. In particular, we have
\[  W_m^\pm e_3 = (J_m^\pm e_1\mp J_m^\pm e_2)e_3 = \mp W_m^\mp. \]

\begin{lemma} \label{lemm:decomp}
  (i) Let $F\in\M(\H)$. Then there exist $f,g\in\M(\A)$ such that
  $F=f+g e_3$. (ii) If $f+ge_3=\tilde f+\tilde ge_3$ with
  $f,g,\tilde f,\tilde g\in\M(\A)$, then $f-\tilde f,g-\tilde g\in \MC(\A)$.
\end{lemma}

\begin{proof}
  (i) Using the operator $\Psi$ of \eqref{eq:defPsi}, construct
   $\tilde f=\Psi F_0,\ g=\Psi F_3$ in $\M(\A)$. Let
  $\varphi = F - (\tilde f + ge_3)$. Then
  \[ \varphi = (F_0-\tilde f_0) + (F_1-\tilde f_1-g_2)e_1 +
    (F_2-\tilde f_2 +g_1)e_2 + (F_3 -g_0)e_3. \] Since $\tilde f_0=F_0$,
  $g_0=F_3$, we have $\varphi\in\MC(\A)$, and $F=f+ge_3$ where
  $f=\tilde f + \varphi$.

  (ii) Since $(f-\tilde f) +(g-\tilde g)e_3=0\in\MC(\A)$, it follows
  from Lemma \ref{lemm:semiunique} that
  $f-\tilde f,g-\tilde g\in \MC(\A)$.
\end{proof}

We would like to find a canonical choice of $f$ in the decomposition
$F=f+ge_3$ of Lemma \ref{lemm:decomp}. Something along these lines was
carried out in \cite{Cac2004} to construct a basis for $\H$-valued
monogenic homogeneous polynomials in a ball, combining $\A$-valued
polynomials with other $\A$-valued polynomials multiplied by
$e_3$. A different application of this idea appears in \cite{Morais2013}. To do something similar for the torus, we need a basis for $\M(\A)$ having a subbasis
of monogenic constants.

We will need the following fact. The  representations of the harmonic functions $1,x_0$ in terms of toroidal harmonics are
\begin{align}  
  1 =\frac{\sqrt{2}}{\pi}  \sum_{n=0}^\infty (1+\delta_{0,n}) I\ind{n,0}{+,+},  \quad
  x_0 = \frac{4\sqrt{2}}{\pi} \sum_{n=1}^\infty
 n\, I\ind{n,0}{-,+}  .\label{eq:1x0fromI} 
\end{align}
These formulas are a particular case of \cite[Eq.\ (55)]{Majic}. One may derive the representation of the constant 1 from \eqref{eq:cohldominici} almost trivially, while it is more challenging to obtain the formula for $x_0$ in this way. By \eqref{eq:IfromIstar}, we
recast \eqref{eq:1x0fromI} in terms of the $I\ind{n,m}{*\,\nu,\mu}$ as
\begin{align}  \label{eq:1x0fromI*}
  1 =  \sum _{k=0}^{\infty}  \alpha_k I\ind{k,0}{*\,-,+}, \quad
  x_0 =  \sum _{k=0}^{\infty} \beta_k I\ind{k,0}{*\,-,+} ,
\end{align}
where
\begin{align*}
  \alpha_k =    \frac{\sqrt{2}}{\pi}
    \sum_{n=k}^\infty (1+\delta_{0,n})i\ind{k,0}{n} ,\quad
  \beta_k =  \frac{4\sqrt{2}}{\pi}\!\!\!
  \sum_{n=(k-1)_+}^\infty \!\!\! n\,i\ind{k,0}{n} .
\end{align*}
We can use the  scalar harmonic function $x_0$
to obtain an expression for $1\in\M(\A)$:
\begin{align} \label{eq:1fromTstar}
  1 =  \partial_0x_0 =  \partial x_0 =
 \partial  \sum_{n=0}^\infty s_n      I\ind{n,0}{*\,-,+}  =
   \sum_{n=0}^\infty s_n T\ind{n+1,0}{+,+}.
\end{align}
 
\begin{proposition} \label{prop:secondbasisA}
  The collection
\begin{align} \label{eq:secondbasisA}
    (\{T\ind{n,m}{\nu,\mu}\}_{n\ge0}\setminus \{T\ind{1,0}{+,+}\}) \ \cup\  \{1\}\ \cup\ \{W_m^\pm\}_{-\infty}^\infty 
\end{align}
is a basis for $\M(\A)$ over $\R$.
\end{proposition}

As always, in this notation all admissible combinations of signs
$(\nu,\mu)$ are intended.

\begin{proof}
We solve   \eqref{eq:1fromTstar} to find
\begin{align*}
  T\ind{1,0}{+,+} = \frac{1}{s_0}( 1 - \sum_{n\ge2}s_{n-1}T\ind{n,0}{+,+} ) .
\end{align*}
Thus the closed span of \eqref{eq:secondbasisA} includes all
of $\M(\A)$ by Theorem  \ref{theo:basisAvalued}. The independence of \eqref{eq:secondbasisA} can
be seen as follows. Suppose that
\begin{align*}
   \sum_{(n,m,\nu,\mu)\not=(1,0,+,+)} \!\!\!\!\!\!\!\!\!\! a\ind{n,m}{\nu,\mu}T\ind{n,m}{\nu,\mu} \ \ \  + \ \ b_0 \ \  + \sum_{m=-\infty}^\infty c_m^\mu W_m^\mu \ \  \ =\ 0.
\end{align*}
The scalar part is 
\begin{align} \label{eq:scalarpart0}
  \sum_{m,\nu,\mu}a\ind{0,m}{\nu,\mu}I\ind{0,m}{\nu,\mu} \ \ \ + \!\!\!\!\!\!\!
  \sum_{\substack{n\ge 1\\(n,m,\nu,\mu)\not=(1,0,+,+)}} \!\!\!\!\!\!\!\!\!\!\! a\ind{n,m}{\nu,\mu}\kappa_{n,m}^{n-1}I\ind{n,m}{*\,\nu,\mu} \ \ \  + \ \ b_0 \ = \ 0 .
\end{align}
If $b_0\not=0$, then we would have
\begin{align*}
  1 = -\frac{1}{b_0} \bigg( \sum_{m,\nu,\mu}a\ind{0,m}{\nu,\mu}I\ind{0,m}{\nu,\mu} \ \ \ + \!\!\!\!\!\!\!\!\!\!\!
  \sum_{\substack{n\ge 1\\(n,m,\nu,\mu)\not=(1,0,+,+)}} \!\!\!\!\!\!\!\!\!\!\! a\ind{n,m}{\nu,\mu}\kappa_{n,m}^{n-1} I\ind{n,m}{*\,\nu,\mu} \bigg).
\end{align*}
This is a series of toroidal harmonics which does not include $I\ind{1,0}{*\,+,+}$, contradicting the unique representation \eqref{eq:1x0fromI*}. Therefore $b_0=0$.
Now \eqref{eq:scalarpart0} is reduced to
\begin{align*}
   \sum_{(n,m,\nu,\mu)\not=(1,0,+,+)} \!\!\!\!\!\!\!\!\!\! a\ind{n,m}{\nu,\mu}T\ind{n,m}{\nu,\mu} \ \ \  + \ \ \sum_{m=-\infty}^\infty c_m^\mu W_m^\mu \ \  \ =\ 0 ,
\end{align*}
and by the independence part of Theorem \ref{theo:basisAvalued}, we conclude that $a\ind{n,m}{\nu,\mu}T\ind{n,m}{\nu,\mu}=0$ and $c_m^\mu=0$.
Therefore the proposed basis is linearly independent as claimed.  
\end{proof}

Now we give our main result.

\begin{theorem}
  The set
  \begin{align} \label{eq:basisH}
    \{T\ind{n,m}{\nu,\mu}\}^{\prime}\, \cup\, \{1\}
    \, \cup\, \{W_m^\mu\}_{-\infty}^\infty   \, \cup\,
    \{T\ind{n,m}{\nu,\mu}e_3\}^{\prime} \, \cup\, \{1e_3\}
  \end{align}
  is a basis for the Hilbert space $\M(\H)$ over $\R$, where $\{\,\}^{\prime}$ excludes the index $(n,m,\nu,\mu) = (1,0,+,+)$.
\end{theorem}

\begin{proof}
  Let us write
  \begin{align*}
    E_{\rm T}' = \Span(\{T\ind{n,m}{\nu,\mu}\}') , \quad E_{\rm W} = \Span \{W_m^\mu\}_{-\infty}^\infty  .
  \end{align*}
  First we show that \eqref{eq:basisH} generates $\M(\H)$.
  Let $F\in\M(\H)$. By Lemma \ref{lemm:decomp}, we can write
  $F = f + ge_3$ for some $f,g\in\M(\A)$. By Proposition \ref{prop:secondbasisA}, $f,g\in E_{\rm T}'+\R+E_{\rm W}$.
Thus $ge_3\in E_{\rm T}'e_3+\R e_3+E_{\rm W}e_3= E_{\rm T}'e_3+\R e_3+E_{\rm W}$ since $W_m^\pm \,e_3 = \mp W_m^\mp$. It then follows that
  \begin{align*}
    F \in  E_{\rm T}'+\R+E_{\rm W}+  E_{\rm T}'e_3 + \R e_3
  \end{align*}
  as claimed.
  
  It only remains to show that \eqref{eq:basisH} is linearly independent. Suppose that
\begin{align*}
  {\sum}' a\ind{n,m}{\nu,\mu} T\ind{n,m}{\nu,\mu} + \ b_0 \ + \  \sum_0^\infty c_m^\mu W_m^\mu\ +
 {\sum}' \hat a\ind{n,m}{\nu,\mu} T\ind{n,m}{\nu,\mu}e_3 \ + \ \hat b_0\,e_3 \
   =0
\end{align*}
with real coefficients, where ${\sum}'$ excludes the index $(n,m,\nu,\mu) = (1,0,+,+)$. Consider the following elements of $\M(\A)$:
\begin{align*}
  f &=   {\sum}' a\ind{n,m}{\nu,\mu} T\ind{n,m}{\nu,\mu} \ + \ b_0 \ + \  \sum_0^\infty  c_m^\mu W_m^\mu, \\
  g &=  {\sum}' \hat a\ind{n,m}{\nu,\mu} T\ind{n,m}{\nu,\mu} \  + \ \hat b_0 .
\end{align*}
Since  $f + ge_3=0$,  by Lemma \ref{lemm:semiunique} we have $f,g\in\MC(\A)$. 
Since $b_0 + \sum c_m^\mu W_m^\mu\in\MC(\A) $, necessarily also
\begin{align*}
  {\sum}' a\ind{n,m}{\nu,\mu} T\ind{n,m}{\nu,\mu}\in\MC(\A) = \Span(\{1\}\cup\{W_m^\pm\}) .
\end{align*}
By Proposition
\ref{prop:secondbasisA}, $\{1\}\cup\{W_m^\pm\}$
is linearly independent of  $\{T\ind{n,m}{\nu,\mu}\}^{\prime}$, so $a\ind{n,m}{\nu,\mu}=0$.  Since $g$ is a monogenic constant, $\hat a\ind{n,m}{\nu,\mu}=0$. Consequently, 
\begin{align*}
  b_0 + \  \sum_0^\infty c_m^\mu W_m^\mu +\hat b_0e_3 =0 .
\end{align*}
Since the scalar and $e_3$ components of this expression vanish, $b_0=\hat b_0=0$. Then
$\sum_0^\infty c_m^\mu W_m^\mu=0$, and finally by linear independence of the
$W_m^\mu$, $c_m^\mu=0$.  This concludes the proof of the linear
independence, so \eqref{eq:basisH} is indeed a basis for $\M(\H)$.
\end{proof}


\section{Discussion}

On a ball, the construction of a convenient basis for the space of monogenic functions is relatively straightforward, and it is known \cite{MoraisHabilitation2021} that for interior prolate and oblate spheroids, it can be carried out in the same way, the only essential difference being that the polynomial basis elements are no longer homogeneous. We have found that an analogous construction for the torus is much more complex. The Appell property for the basic Fueter operator $\partial$ goes in the opposite direction, with the consequence that the basic monogenics of index $n=0$ must be constructed in a completely different way from the simple application of $\partial$ to scalar harmonics.  The non-trivial topology of the torus adds an additional complication.

We have identified the $\A$-valued monogenic constants in this space, constructed a basis for the $\A$-valued monogenics, and then a basis for the $\H$-valued monogenic functions on the torus. Because the bases are comprised of three essentially different types of functions, the task of orthogonalizing them appears to be a project that will depend on using new ideas.


\bigskip
\textbf{\large Acknowledgments}

Z. Ashtab was supported by CONACyT (now CONAHCyT), Mexico. J. Morais' work was partially supported by the Asociación Mexicana de Cultura, A.\ C.


\newcommand{\authorlist}[1]{#1}
\newcommand{\booktitle}[1]{\textit{#1}}
\newcommand{\articletitle}[1]{#1}
\newcommand{\journaltitle}[1]{\textit{#1}}
\newcommand{\volnum}[1]{\textbf{#1}}

\end{document}